\documentclass[a4paper,11pt]{article}
\author{Jelena~Jovanovi\'c\\ \emph{University of Belgrade, Faculty of Mathematics }}
\title{On terms describing omitting unary and affine types -- improved}
\usepackage{amsmath}
\usepackage{amssymb}
\usepackage{amsthm}
\usepackage{verbatim}
\date{}
\theoremstyle{plain}
\newtheorem{teorema}{Theorem}[section]
\newtheorem{definicija}[teorema]{Definition}
\newtheorem{tvrdjenje}[teorema]{Proposition}
\newtheorem{fakt}[teorema]{Fact}

\begin{document}
\maketitle
\begin{abstract}\footnote{2010 Mathematics Subject Classification: Primary 08B05;\\Keywords: Variety, Hobby-McKenzie types, Omitting types;}
In this paper we examine the possibility of describing omitting types 1 and 2 by two at most ternary terms and any number of linear identities. All possible cases of systems of linear identities on two at most ternary terms are being analyzed, and it is shown that only a single one of these systems might describe omitting types 1 and 2. However, we do not resolve whether it actually describes omitting mentioned types, but only prove that it implies this property, so this question is left for further examination.
\end{abstract}

\section{Introduction}

This paper deals with syntactical conditions equivalent to congruence meet-semidistributivity of a locally finite variety of algebras. We first endeavor to explain the significance of the concept under study in this paper. Congruence meet-semidistributivity of a locally finite variety has already been characterized by D. Hobby and R. McKenzie (\cite{hm}) as equivalent to having no covers of types {\bf 1} or {\bf 2} in the congruence lattices of finite algebras in the variety, and also by a syntactical Mal'cev condition due to R. Willard (\cite{rossw}). It is a condition which has seen broad use, as much of the theory which holds true for congruence distributive varieties, also holds in congruence meet-semidistributive case, though the proofs are often much harder. To mention some better-known examples, Park's Conjecture (\cite{park}) was proved to hold in congruence meet-semidistributive case by Willard (\cite{rossw}), extending the congruence distributive proof due to Baker (\cite{baker}), while Quackenbush's Conjecture (\cite{quack}), which holds trivially in congruence distributive case due to J\'onsson's Lemma (\cite{bjarni}), was proved in the congruence meet-semidistributive case by Kearnes and Willard (\cite{kearnesross}). Recently, the research in the Constraint Satisfaction Problem has shown that congruence meet-semidistributivity of the variety generated by the algebra of compatible operations is equivalent to the condition that the efficient algorithm called 'localconsistency checking' to faithfully solve the Constraint Satisfaction Problem. The property 'local consistency checking works here' is called 'bounded width' by the researchers working in Constraint Satisfaction Problem, and the reader can find a wealth of literature on the concept. This result due to L. Barto and M. Kozik (\cite{bartokozik}) is certainly among the strongest known partial results for the Dichotomy Conjecture, and probably the hardest to be proven so far.

Characterization of various semantical properties of all algebras and/or their congruence lattices in a variety by equivalent syntactical conditions was started by A. I. Mal'cev in (\cite{malcev}). Because of that, the properties which can be so characterized are called Mal'cev properties, and the syntactical conditions equivalent to them Mal'cev conditions. There is a subclass of the {\em strong} Mal'cev conditions saying that a property is equivalent to all algebras in the variety having a fixed number of term operations of fixed arities and which satisfy a fixed package of equations (like the original Mal'cev condition for congruence permutability). On the other hand, a usual Mal'cev property is equivalent to satisfying one of (equivalently, all but finitely many of) a countably infinit sequence of strong Mal'cev conditions in which each member implies the next one (meaning that the properties are increasingly more general), like in the case of J\'onsson's condition for congruence distributivity (\cite{bjarni}). Obviously, strong Mal'cev conditions are preferable, if available, as then we can use the operations in a computer search, but there are properties which are Mal'cev properties but are proved to not be strong Mal'cev properties. The condition most commonly used for congruence meet-semidistributivity of a variety (not necessarily locally finite), until recently, was the one proved by Willard (\cite{rossw}), but the research in the Constraint Satisfaction Problem has recently uncovered that the congruence meet-semidistributivity of a locally finite variety is a strong Mal'cev property. The best, i. e. syntactically strongest we know of is due to M. Kozik (\cite{hhhm}). In this paper we try to see if it is also the best possible. We identify three candidate conditions, all of which imply congruence meet-semidistributivity and are syntactically stronger and with fewer operations and/or of smaller arity than the condition proved in \cite{hhhm}. We prove that either one of the three conditions we found is indeed equivalent to congruence meet-semidistributivity, or the condition from \cite{hhhm} is the best possible. Which of these alternatives is true we have not managed to ascertain and leave open for future research.

\section{Background}

In this paper an {\em algebra} denotes a structure ${\bf A}=(A,F^{\bf A})$, where $F$ is a signature, or language, consisting only of operation symbols of various arities, $A$ is a nonempty set, and for each symbol $f\in F$ of arity $k$ the corresponding element $f^{\bf A}\in F^{\bf A}$ is a mapping $f^{\bf A}:A^k\rightarrow A$. The set of {\em term operations} of ${\bf A}$ is the set of all operations obtained from $F^{\bf A}$ and projection operations via finitely many compositions. All algebras of the same signature which identically satisfy a set of equations are called a variety. An algebra ${\bf A}$ is locally finite if for any finite subset $X$ of $A$, the set of all results of term operations applied to elements of $X$ is also finite. A variety is locally finite if every algebra in it is.

There is a natural connection between operations and relations on the same set. It says that a ($k$-ary) relation and an ($n$-ary) operation are compatible if for any $n$ vectors from the relation, the vector obtained by pointwise application of the operation is again in the relation. The classic results of universal algebra often connect the properties of the compatible equivalence relations, which form a lattice under inclusion called the congruence lattice, and other properties of algebras. In this paper we investigate the meet-semidistributivity of the congruence lattices of all algebras in a variety, which is the lattice implication $x\wedge z = y\wedge z \Rightarrow (x\vee y)\wedge z = x\wedge z$. An equivalent condition of the congruence meet-semidistributivity of a locally finite variety is omitting types of covers ${\bf 1}$ and ${\bf 2}$ in finite algebras of the variety. This concept will be explained further in the text. For any other definitions and basic results which are not found in this introductory part, the reader is referred to \cite{hhhhm} for basic universal algebra and \cite{hm} for tame congruence theory.

\begin{definicija}
Let $\mathbf{A}$ be a finite algebra and $\alpha$ a minimal congruence of $\mathbf{A}$ (i.e. $0_{\mathbf{A}} <  \alpha$ and if $\beta$ is a congruence of $\mathbf{A}$ with $0_{\mathbf{A}} <  \beta \le \alpha $ then $\beta = \alpha$.)
\begin{itemize}
\item An $\alpha$--minimal set of $\mathbf{A}$ is a subset $U$ of $\mathbf{A}$ that satisfies following two conditions:
\begin{itemize}
\item[-]$U = p(\mathbf{A})$ for some unary polynomial $p(x)$ of $\mathbf{A}$ that is not constant on at least one $\alpha$--class
\item[-] with respect to containment, $U$ is minimal having this property.
\end{itemize}
\item An $\alpha$--neighbourhood (or $\alpha$--trace) of $\mathbf{A}$ is a subset $N$ of $\mathbf{A}$ such that:
\begin{itemize}
\item[-] $N = U \cap (a /_{\alpha})$ for some $\alpha$--minimal set $U$ and $\alpha$--class $a /_{\alpha}$
\item[-] $\vert N \vert > 1$.
\end{itemize}
\end{itemize}
\end{definicija}
We can easily see that a given $\alpha$--minimal set $U$ must contain at least one, and possibly more, $\alpha$--neighbourhoods.The union of all $\alpha$--neighbourhoods in $U$ is called the body of $U$, and the remaining elements of $U$ form the tail of $U$. What is important here is that algebra $\mathbf{A}$ induces uniform structures on all its $\alpha$--neighbourhoods, meaning they (the structures induced) all belong to the same of five possible types. Let us now define an induced structure.
\begin{definicija}
Let $\mathbf{A}$ be an algebra and $U \subseteq \mathbf{A}$. The algebra induced by $\mathbf{A}$ on $U$ is the algebra with universe $U$ whose basic operations consist of the restriction to $U$ of all polynomials of $\mathbf{A}$  under which $U$ is closed. We denote this induced algebra by $\mathbf{A} \vert_{U}$.
\end{definicija}
\begin{teorema}
Let $\mathbf{A}$ be a finite algebra and $\alpha$ a minimal congruence of $\mathbf{A}$.
\begin{itemize}
\item If $U$ and $V$ are $\alpha$--minimal sets then $\mathbf{A} \vert_{U}$ and $\mathbf{A} \vert_{V}$ are isomorphic and in fact there is a polynomial $p(x)$ that maps $U$ bijectively onto $V$.
\item If $N$ and $M$ are $\alpha$--neighbourhoods then $\mathbf{A} \vert_{N}$ and $\mathbf{A} \vert_{M}$ are isomorphic via the restriction of some polynomial of $\mathbf{A}$.
\item If $N$ is $\alpha$--neighbourhood then $\mathbf{A} \vert_{N}$ is polynomially equivalent to one of:
\begin{enumerate}
\item A unary algebra whose basic operations are all permutations (unary type);
\item A one--dimensional vector space over some finite field (affine type);
\item A $2$--element boolean algebra (boolean type);
\item A $2$--element lattice (lattice type);
\item A $2$--element semilattice (semilattice type);
\end{enumerate}
\end{itemize}
\end{teorema}
\begin{proof}
The theorem in this form is given in \cite{nm}, and the proof can be found in \cite{hm}.
\end{proof}
The previous theorem allows us to assign a type to each minimal congruence $\alpha$ of an algebra according to the behaviour of the $\alpha$--neighbourhoods (for example, a minimal congruence whose $\alpha$--neighbourhoods are polynomially equivalent to a vector space is said to have affine type or type $2$).

Taking this idea one step further, given a pair of congruences $(\alpha, \beta)$ of $\mathbf{A}$ with $\beta$ covering $\alpha$ (i.e. $\alpha < \beta$ and there are no congruences of $\mathbf{A}$ strictly between the two), one can form the quotient algebra $\mathbf{A}/_{\alpha}$, and then consider the congruence $\beta /_{\alpha} = \{(a/_{\alpha}, b/_{\alpha}):(a,b) \in \beta \}$. Since $\beta$ covers $\alpha$ in the congruence lattice of $\mathbf{A}$,  $\beta /_{\alpha}$ is a minimal congruence of $\mathbf{A}/_{\alpha}$, so it can be assigned one of the five types. In this way we can assign to each covering pair of congruences of $\mathbf{A}$ a type (unary, affine, boolean, lattice, semilattice, or 1, 2, 3, 4, 5 respectively). Therefore, going through all covering pairs of congruences of this algebra we obtain a set of types, so--called typeset of $\mathbf{A}$, denoted by $typ \{\mathbf{A}\}$. Also, for $\mathcal{K}$  a class of algebras, the typeset of $\mathcal{K}$ is defined to be the union of all the typesets of its finite members, denoted by $typ \{\mathcal{K} \}$.

A finite algebra or a class of algebras is said to omit a certain type if that type does not appear in its typeset. For locally finite varieties omitting certain types can be characterized by Maltsev conditions, i.e. by the existence of certain terms that satisfy certain linear identities, and there are quite a few results on this so far. We shall present two of them concerning omitting types 1 and 2.
\begin{definicija}
An n--ary term $t$, for $n > 1$, is a near--unanimity term for an algebra $\mathbf{A}$ if the identities $t(x,x,\dots ,x,y) \approx t(x,x,\dots ,y,x) \approx  \dots  \approx t(x,y,\dots ,x,x) \approx t(y,x,\dots ,x,x) \approx x $ hold in $\mathbf{A}$.
\end{definicija}
\begin{definicija}
An n--ary term $t$, for $n > 1$, is a weak near--unanimity term for an algebra $\mathbf{A}$ if it is idempotent and the identities $t(x,x,\dots ,x,y) \approx t(x,x,\dots ,y,x) \approx  \dots  \approx t(x,y,\dots ,x,x) \approx t(y,x,\dots ,x,x)$ hold in $\mathbf{A}$.
\end{definicija}
\begin{teorema}

A locally finite variety $\mathcal{V}$ omits the unary and affine types (i.e. types 1 and 2) if and only if there is some $N > 0$ such that for all $k > N$, $\mathcal{V}$ has a weak near--unanimity term of arity $k$.\label{th}

\end{teorema}
\begin{proof} The proof can be found in \cite{hhm}.
\end{proof}
\begin{teorema}
A locally finite variety $\mathcal{V}$ omits the unary and affine types if and only if it has 3--ary and 4--ary weak near--unanimity terms, $v$ and $w$ respectively, that satisfy the identity $v(y,x,x) \approx w(y,x,x,x)$.
\end{teorema}
\begin{proof}
The proof can be found in \cite{hhhm}.
\end{proof}

Therefore, omitting types 1 and 2 for a locally finite variety can be described by linear identities on 3--ary and a 4--ary term, both idempotent. In this paper we examine whether the same can be done by two at most 3--ary idempotent terms. It is sufficient to focus our attention to idempotent terms only, for the reasons that are explained in detail in \cite{hhm}. We shall use two particular algebras for our examination, both of them omitting types 1 and 2.

\section{ Examples of algebras generating varieties that omit unary and affine types } \label{section two}
\noindent $\mathbf{Example \  1}$\\

\noindent Let $\mathbf{B}=\langle \ \{\ \!0 \ ,\ 1 \} \ ,\ \land  \  \rangle$ \label{example 1} be the semilattice with two elements (i.e. $\land$ stands for a commutative, associative and idempotent binary operation). This algebra generates a variety that omits types 1 and 2, according to Theorem 2.7 from above (weak near--unanimity terms being $v(x,y,z)\approx x\land y \land z $ and $w(x,y,z,u)\approx x\land y\land z\land u$). \\

\noindent $\mathbf{Example \  2}$\\

\noindent Let $\mathbf{A}$ be a finite algebra with at least two elements and a single idempotent basic operation $ f(x_1,x_2,x_3)$, which is a ternary near--unanimity term (i.e. a majority term): \begin{equation*} f(x,x,y) \approx f(x,y,x) \approx f(y,x,x) \approx x \end{equation*} In case no arguments are equal, we can define $f$ like this:
\begin{equation*} f(a,b,c) = a,\  \text{for all}\  a,b,c \in \mathbf{A}\   \text{and}\  a\ne b,\  b \ne c,\  c\ne a
\end{equation*}
We shall prove now that algebra $\mathbf{A}$ generates a variety that omits types 1 and 2:

Let $g(x,y,w,z)$ be a $4$--ary term of this algebra defined by: $g(x,y,w,z) \approx f(x,y,f(x,w,z))$, $f$ being the basic operation. It is easy to check that $g$ is a weak near--unanimity term, and the identity $g(y,x,x,x) \approx  f(y,x,x)$ holds, so algebra $\mathbf{A}$ generates a variety that omits unary and affine types according to Theorem 2.7 from above.
\vspace{0.3 cm}

This algebra has some interesting properties:
\begin{enumerate}
\item every binary term--operation $t$ of $\mathbf{A}$ must satisfy one of these two identities:\label{binary term} \\$t(x,y)\approx x$ \\ $t(x,y)\approx y$;\\
\noindent In other words the only binary term--operations on $\mathbf{A}$  are projections $\pi_1 , \pi_2$;
\begin{proof}
We shall prove the statement by induction on the complexity of the term $t$:
\begin{itemize}
\item[-] if $t(x,y)$ is a projection the statement holds
\item[-] if $t(x,y)\approx f(t_1(x,y), t_2(x,y), t_3(x,y))$, where $t_1,t_2,t_3$ are less complex binary terms, then these three are projections by the induction hypothesis, therefore at least two of them are equal, so $t$ must be a projection too.
\end{itemize}
\end{proof}
\item every ternary term--operation $p$  of $\mathbf{A}$ satisfies exactly one of the following:\label{ternary term}\\$p(x,y,z)\approx x$\\ $p(x,y,z)\approx y$\\ $p(x,y,z)\approx z$\\ $p(x,x,y)\approx p(x,y,x)\approx p(y,x,x)\approx x$;\\

\vspace{0.1 cm}
\noindent This means $p$ is either one of the projections $\pi_1,\pi_2,\pi_3$  or a majority term--operation, that is there are no other ternary term--operations except for these four kinds.
\begin{proof}
We prove the second statement also by induction on the complexity of the term--function $p$:
\begin{itemize}
\item if $p(x,y,z)$ is a projection or the basic operation $f(x,y,z)$, the statement holds
\item if $p(x,y,z)\approx f(p_1(x,y,z), p_2(x,y,z), p_3(x,y,z))$,where $p_1,p_2,p_3$ are less complex ternary terms, then each of these three is either a projection or a  majority term by the induction hypothesis, so we have the following cases:
\begin{itemize}
\item[-] if at least two of  $p_1,p_2,p_3$ are majority terms, then $p$ is also a majority term;
\item[-] if exactly one of  $p_1,p_2,p_3$ is a majority term and remaining two are the same projection $\pi_j \ \text{for some}\  j \in \{1,2,3\}$ then $p$ is also a projection $\pi_j$;
\item[-] if exactly one of  $p_1,p_2,p_3$ is a majority term and remaining two are projections $\pi_i,\pi_j \ \text{for some}\ \\ i,j \in \{1,2,3,\}\ \  \text{and}\ \  i \ne j$\  then $p$ is a majority term;
\item[-] if the terms $p_1,p_2,p_3$ are projections $\pi_1,\pi_2,\pi_3$ (in whichever order) then $p(x,y,z)$ is $f(x,y,z)$ up to the permutation of variables, which is still a majority term;
\item[-] if the terms $p_1,p_2,p_3$ are projections $\pi_i,\pi_i,\pi_j$ (again in whichever order) $ \text{for some}\  i,j \in \{1,2,3\} \ \text{and}\  i \ne j$  then $p$ is a projection $\pi_i$;
\end{itemize}
\end{itemize}
\end{proof}
\end{enumerate}

\vspace{0.2 cm}
Now that we have listed the examples needed, let us notice that any system of (linear) identities possibly describing omitting types 1 and 2 (including any number of terms) must hold in algebras $\mathbf{B}$ and $\mathbf{A}$ from examples 1 and 2 respectively. We shall make use of this fact in the rest of the paper.

In section 3 we discuss systems of linear identities on a single binary term, two binary terms, a single ternary term and a binary and a ternary term. We prove that none of these systems describes omitting types 1 and 2 ( in fact none of them even implies omitting these two types).

In section 4 systems on two ternary terms are being discussed, and we prove that there are only three of them that could possibly describe omitting types 1 and 2 ( all three imply omitting these two types). As mentioned in the abstract, we do not resolve whether any of them actually describes this property.

\section{Systems of linear identities on a single binary term, two binary terms, a single ternary term and a binary and a ternary term}\label{section2}

As we have already mentioned, in this section we discuss systems of linear identities on a single binary term, two binary terms, a single ternary term and a binary and a ternary term, each of these cases being analyzed in a separate subsection. We shall prove here that none of these systems describes omitting types 1 and 2 (in fact none of them even implies omitting these two types).
\subsection {A single binary idempotent term} \label{singlebinary}
\vspace{0.1cm}

Let $t(x,y)$ be an idempotent binary term. If a system of linear identities on $t(x,y)$ (a single identity or more) describes omitting types 1 and 2, it must hold in algebra $\mathbf{A}$ from example 2, which means $t(x,y)$ has to be a projection map in $\mathbf{A}$. So, the system considered must allow $t(x,y)$ to be a projection map and it must not yield a trivial variety (algebra), but such a system holds in every algebra, therefore does not describe omitting types 1 and 2.

\vspace{0.1cm}

\noindent We can conclude that omitting types 1 and 2 cannot be described by a single binary idempotent term (using any number of identities).

\subsection{Two binary idempotent terms}\label{twobinary}

\vspace{0.1cm}

Let $t(x,y)$ and $s(x,y)$ be idempotent binary terms. If a system of identities on $t(x,y)$ and $s(x,y)$ describes omitting types 1 and 2, it must hold in algebra $\mathbf{A}$ from example 2, which means both $t$ and $s$ have to be projection maps. This means the identities of the system considered must allow both $t$ and $s$ to be projection maps, but these exist in every algebra, so the system cannot describe omitting types 1 and 2.

\vspace{0.3 cm}

\noindent From the previous we conclude that omitting types 1 and 2 cannot be described by two binary idempotent terms (using any number of identities).

\subsection{A single ternary idempotent term}\label{singleternary}

In this subsection we prove that omitting types 1 and 2 cannot be described by any number of linear identities on a single ternary idempotent term.
\noindent Previously, let us consider a specific reduct of a module that we use in the proof-- a full idempotent reduct of a module over $\mathbb{Z}_5$ (this is an algebraic structure obtained from a module over $\mathbb{Z}_5$ by taking into consideration only the idempotent term--operations of the module and all such term--operations):

In a full idempotent reduct of a module over $\mathbb{Z}_5$ a ternary term $p(x,y,z)$ must satisfy one of the following identities:
\begin{equation}
\begin{array}{lll}
p(x,y,z)\approx x & p(x,y,z)\approx y & p(x,y,z)\approx z\\
p(x,y,z)\approx 4x + 2y & p(x,y,z)\approx 4x + 2z & p(x,y,z)\approx 4y + 2z\\
p(x,y,z)\approx 2x + 4y & p(x,y,z)\approx 2x + 4z & p(x,y,z)\approx 2x + 4z\\
p(x,y,z)\approx 3x + 3z & p(x,y,z)\approx 3x + 3y & p(x,y,z)\approx 3y + 3z \\
p(x,y,z)\approx x + 2y + 3z & p(x,y,z)\approx x + 3y + 2z & p(x,y,z)\approx 2x + y + 3z\\
p(x,y,z)\approx 2x + 3y + z & p(x,y,z)\approx 3x + 2y + z & p(x,y,z)\approx 3x + y + 2z\\
p(x,y,z)\approx 4x + y + z& p(x,y,z)\approx x + y + 4z & p(x,y,z)\approx x + 4y + z\\
p(x,y,z)\approx 2x + 2y + 2z\\
\end{array}
\end{equation}There are no other ternary terms in this reduct.
\vspace{0.1cm}

\noindent Now we can discuss systems of linear identities on a single ternary idempotent term.

\vspace{0.1cm}

\begin{fakt} Suppose a system of identities on $p(x,y,z)$ describes omitting types 1 and 2, and let us denote it by $\sigma$. Then the system $\sigma$ has to hold in algebras $\mathbf{B}$ and $\mathbf{A}$ from examples 1 and 2 respectively, and it must not hold in any full idempotent reduct of a module over a finite ring (theorem 8 in \cite{nm}, or more detailed in \cite{nnm}).
\end{fakt}

\noindent Based on this fact we can state the following:
\begin{itemize}
\item if identities of the system $\sigma$ allow $p$ to be defined as a projection map in algebra $\mathbf{A}$ (any projection map) then $p$ can be defined as a projection map in any algebra, so the system does not describe omitting types 1 and 2. Therefore the identities of the system must have forms that allow $p$ to be a majority term and only a majority term in $\mathbf{A}$.
\item if there is an identity of the form $p(x,y,z)\approx p(u,v,w)$ in the system $\sigma$, then $\{x,y,z\} = \{u,v,w\}$, i.e. $(u,v,w)$ is a permutation of $(x,y,z)$, for otherwise $p$ could not be a majority term in $\mathbf{A}$.
\item considering identities with less than three variables on either side: if the left hand side of an identity is some of $p(x,x,y), p(x,y,x), p(y,x,x)$, then on the right there has to be either $x$ alone, or one of the terms $p(x,x,y), p(x,y,x), p(y,x,x),p(x,x,z),p(z,x,x),p(x,z,x)$(again for the same reason, $p$ being necessarily a majority term in $\mathbf{A}$).
\item if the system $\sigma$ contains only identities having variables $x,y,z$ on both sides and/or identities with $x,x,y$ on both sides (that is $x$ occurs twice, $y$ once on both sides) then the system holds in a full idempotent reduct of a module over $\mathbb{Z}_5$ (and therefore does not describe omitting types 1 and 2), for we can define $p$ to be $2x + 2y + 2z$ in this reduct. This means $\sigma$ has to include an identity with $x,x,y$ on the left and $x,x,z$ on the right (up to a permutation of these variables, of course), or $x,x,y$ on the left and $x$ alone on the right.
\item since the system $\sigma$ has to hold in algebra $\mathbf{B}$ (example 1), from the previous item we can conclude that $p$ has to be a binary term in $\mathbf{B}$ (it cannot be a projection map for the system does not allow that). Now, if the system allows $p(x,y,z)$ to be defined as a binary term in $\mathbf{B}$, i.e. one of the terms $x \land y$, $y \land z$, $x \land z$, then it also allows $p$ to be one of the terms $3x+3y$, $3y+3z$, $3x+3z$ in a full idempotent reduct of a module over $\mathbb{Z}_5$. This means $\sigma$ does not describe omitting types 1 and 2.
\end{itemize}

\noindent This proves that no system on $p(x, y, z)$ satisfies the necessary conditions for describing omitting types 1 and 2 given above ( i.e. to hold in algebras $\mathbf{B}$ and $\mathbf{A}$ and not to hold in any full idempotent reduct of a module over a finite ring).

\subsection{A binary idempotent term and a ternary idempotent term}\label{bin_ter}

\vspace{0.1cm}

\noindent In this subsection we shall discuss whether a system of any number of linear identities on $t(x,y)$ and $p(x,y,z)$ can describe omitting types 1 and 2 ($t$ and $p$ both being idempotent terms).

\vspace{0.1 cm}
\noindent Suppose we have a system on $t$ and $p$ describing omitting types 1 and 2, and let us denote it by $\tau$. It has to hold in $\mathbf{A}$, so $t$ has to be a projection map, and $p$ either a projection map or a majority term in this algebra (this is proved in section \ref{section two}). If the system $\tau$ allows both $t$ and $p$ to be projections in $\mathbf{A}$, it holds in any algebra ($t$ and $p$ being the same projection maps as in $\mathbf{A}$), so it cannot describe omitting types 1 and 2. Therefore $\tau$ must hold in $\mathbf{A}$ only for $p$ being a majority term (and $t$ a projection map, of course). Further more, algebra $\mathbf{B}$ has to satisfy the system also, so let us analyze possible cases:
\begin{itemize}
\item both $t$ and $p$ can be defined as projection maps in $\mathbf{B}$ -- this is impossible because if this were the case then both terms could be defined as projection maps in $\mathbf{A}$, and we have already excluded that.
\item $t$ can be defined as a projection map and $p$ as a binary term (i.e. one of the $x \land y$, $y \land z$, $x \land z$) in $\mathbf{B}$ -- then we can define $t$ to be the same projection map, and $p$ to be one of the terms $3x+3y$, $3y+3z$, $3x+3z$ in a full idempotent reduct of a module over $\mathbb{Z}_5$, i.e. $t$ and $p$ exist in this reduct, therefore the system $\tau$ does not describe omitting types 1 and 2.
\item $t$ can be defined as a projection map and $p$ as a ternary term in $\mathbf{B}$ (meaning, of course, that we cannot define $p$ as either a projection map or a binary term in $\mathbf{B}$ )-- in this case the system $\tau$ cannot contain an identity on $t$ and $p$, i.e. all the identities are either only on $t$ or only on $p$. Since the identities on $t$ allow $t$ to be a projection map they can be ignored, so $\tau$ describes omitting types 1 and 2 if and only if remaining identities only on $p$ do the same. This is already proved to be impossible (subsection \ref{singleternary}).
\item $t$ can be defined only as a binary term and $p$ as a projection map in $\mathbf{B}$ -- once again $\tau$ cannot include an identity on $t$ and $p$, i.e. all the identities are either only on $t$ or only on $p$. Since $t$ has to be a projection map in $\mathbf{A}$, the identities on $t$ must allow that, which means $t$ can be defined as the same projection map in $\mathbf{B}$. Therefore this case is impossible.
\item both $t$ and $p$ can be defined as  binary terms in $\mathbf{B}$ (and of course, none of them as a projection map) -- if this is the case it is easily seen that both can be defined as binary terms in a full idempotent reduct of a module over $\mathbb{Z}_5$ (some of the terms $3x+3y$, $3y+3z$, $3x+3z$), so the system $\tau$ does not describe omitting types 1 and 2.
\item $t$ can be defined as a binary term and $p$ as a ternary term in $\mathbf{B}$ (and no other possibilities, as before) -- then there is no identity only on $t$ in $\tau$, since it would have to be this one $t(x,y) \approx t(y,x)$, and it cannot hold in $\mathbf{A}$. Further more, if $t(x,y)$ is on the left, then on the right we have term $p$ with variables $x$ and $y$ only. This allows us to eliminate term $t$ from all the identities except for one, obtaining an equivalent system. Now we can ignore the identity with $t$ (the only one of the form $t(x,y) \approx p(u,v,w)$, where $\{u,v,w\}=\{x,y\}$) and state that the system $\tau$ describes omitting types 1 and 2 if and only if the remaining identities only on $p$ do the same, which is impossible (subsection \ref{singleternary}).
\end{itemize}

\vspace{0.1cm}
\noindent By this we have proved that omitting types 1 and 2 cannot be described by a binary and a ternary term, both idempotent (using any number of linear identities).
\section{Two ternary idempotent terms}\label{stt}
\vspace{0.1 cm}
In the following section we shall discuss systems of linear identities on two ternary terms, both idempotent, and we shall prove that only three of these systems could possibly describe omitting types 1 and 2 (all three imply omitting these two types). However, we do not resolve whether any of them actually describes this property.

\vspace{0.1 cm}

\noindent Let $p(x,y,z)$ and $q(x,y,z)$ be ternary idempotent terms; as before we shall suppose there is a system of linear identities on $p$ and $q$ describing omitting types 1 and 2, and we shall denote it by $\phi$.

\noindent Let us notice the important fact: if the system $\phi$ has no identity on $p$ and $q$, i.e. all its identities are either only on $p$ or only on $q$, then we can apply the conclusion that we came to in subsection \ref{singleternary}: if $p$ exists in algebras $\mathbf{B}$ and $\mathbf{A}$ from examples 1 and 2, then $p$ also exists in a full idempotent reduct of a module over $\mathbb{Z}_5$, and the same holds for $q$. Therefore the system $\phi$ needs to have at least one identity on $p$ and $q$.

\noindent Regarding the fact that $\phi$ has to hold in algebra $\mathbf{A}$ (i.e. terms $p$ and $q$ have to exist in this algebra) there are three possible cases:
\begin{enumerate}
\item $p$ and $q$ can both be projection maps in algebra $\mathbf{A}$ -- then $\phi$ holds in any algebra ($p$ and $q$ being the same projection maps as in $\mathbf{A}$), so there is no need to analyze this case any further.
\item the system $\phi$ allows only one of $p$, $q$ to be a projection map in $\mathbf{A}$, and the other term has to be a majority term in this algebra.
\item the system $\phi$ does not allow either of $p$ and $q$ to be a projection map in $\mathbf{A}$, i.e. both are majority terms in this algebra.
\end{enumerate}
We shall analyze cases 2 and 3 in the following two subsections -- in the subsection \ref{cases12} we deal with case 2, and in subsection \ref{case3} with case 3.
\subsection{ $\mathbf{p}$ and $\mathbf{q}$ are a projection and a majority term in $\mathbf{A}$ } \label{cases12}

\vspace{0.3 cm}

\noindent Suppose the system $\phi$ holds in $\mathbf{A}$ for a projection map and a majority term (case 2 from above). We can assume with no loss of generality that $p$ is $\pi_1$ and $q$ a majority term, therefore these terms satisfy the following in $\mathbf{A}$:
\begin{equation}
\begin{array}{r}
x\approx p(x,x,y)\approx p(x,y,y)\approx p(x,y,x)\approx q(x,x,y)\approx q(x,y,x)\approx q(y,x,x)
\label{sys}
\end{array}
\end{equation}
\noindent It is easily seen that $p$ and $q$ satisfying the system do not exist in $\mathbf{B}$ (example 1), so \eqref{sys} cannot be the system describing omitting types 1 and 2. We shall try to obtain the system mentioned by eliminating some identities from \eqref{sys}, so let us eliminate the first one:
\begin{equation}
\begin{array}{r}
p(x,x,y)\approx p(x,y,y)\approx p(x,y,x) \approx q(x,x,y)\approx q(x,y,x)\approx q(y,x,x)
\end{array}\label{system1}
\end{equation}
\noindent We shall prove here that system \eqref{system1} implies omitting types 1 and 2, because terms $p$, $q$ do not exist in any full idempotent reduct of a module over a finite ring:
\begin{proof}
If term $p$ is projection map $\pi_1$ in a reduct, then $q$ would have to be at most a binary term, which is impossible. Term $p$ cannot be a binary term in any reduct either (we can see this directly from the identities), so it has to be a ternary term, i.e. $\alpha x + \beta y + \gamma z$ for $\alpha + \beta + \gamma =1$ and none of $\alpha, \beta, \gamma$ is zero. Now, from the first identity we obtain $\alpha x + \beta x + \gamma y = \alpha x + \beta y + \gamma y$, which yields $\alpha  + \beta = \alpha$, i.e. $\beta = 0$. Therefore $p$ does not exist in any reduct over a finite ring, and this completes the proof.
\end{proof}
\vspace{0.1cm}
\noindent So, we have obtained the system that implies omitting types 1 and 2, but it is not minimal -- namely, the system given below, obtained from the previous by eliminating identity $p(x,y,y)\approx p(x,y,x)$, also implies omitting types 1 and 2 (this is proven the same way like the above case).
\begin{equation}
\left\{
\begin{array}{r}
p(x,x,y)\approx p(x,y,y)\\
p(x,y,x)\approx q(x,x,y) \approx q(x,y,x) \approx q(y,x,x)
\end{array}\right. \label{system2}
\end{equation}
\noindent Moreover, it can be proved that \eqref{system2} is a minimal system with this property:
\begin{proof}
It'll be sufficient to prove that any set of identities which is a proper subset of \eqref{system2} allows us to define $p$ and $q$ in some reduct of a module over a finite ring.

\vspace{0.4 cm}

\noindent If we eliminate the first identity, we obtain the system:
\begin{equation*}
\begin{array}{r}
p(x,y,x)\approx q(x,x,y) \approx q(x,y,x)\approx q(y,x,x)
\end{array}
\end{equation*}
Now both terms can be defined as $2x + 2y + 2z$ in a reduct over $\mathbb{Z}_5$. So the first identity ($p(x,x,y)\approx p(x,y,y)$) must stay.

\vspace{0.2 cm}

\noindent If the term $p(x,y,x)$ is completely omitted, we have an identity only on $p$ (the first one), and the rest of them are only on $q$, but in this case $p$ can be defined as $\pi_1$ and $q$ as $2x + 2y + 2z$ in a reduct over $\mathbb{Z}_5$. This means we have to keep one of the following identities (i.e. at least one) $p(x,y,x) \approx q(x,x,y), p(x,y,x) \approx q(x,y,x), p(x,y,x) \approx q(y,x,x)$, and we shall consider it the second. Let us go through the cases now:
\begin{itemize}
\item if we have $p(x,y,x) \approx q(x,x,y)$ as the second identity (the first identity being $p(x,x,y)\approx p(x,y,y)$, as explained), obviously we need to add one more identity from the system \eqref{system2} to these two, for if we do not, $p$ and $q$ can both be projection maps (in fact a single identity is all we can add here -- adding any two identities from  \eqref{system2} gives us the whole system \eqref{system2}). We have three options for the third identity:
\begin{equation*}
\left\{
\begin{array}{r}
p(x,x,y)\approx p(x,y,y)\\
p(x,y,x) \approx q(x,x,y)\\
q(x,y,x)\approx q(y,x,x)
\end{array}\right.
\end{equation*}
this system allows $p$ to be $\pi_1$ and $q$ to be $3x + 3y$ in a reduct over $\mathbb{Z}_5$;
\begin{equation*}
\left\{
\begin{array}{r}
p(x,x,y)\approx p(x,y,y)\\
p(x,y,x) \approx q(x,x,y)\approx q(y,x,x)
\end{array}\right.
\end{equation*}
this system allows $p$ to be $\pi_1$ and $q$ to be $\pi_2$ in a reduct over $\mathbb{Z}_5$;
\begin{equation*}
\left\{
\begin{array}{r}
p(x,x,y)\approx p(x,y,y)\\
p(x,y,x) \approx q(x,x,y)\approx q(x,y,x)
\end{array}\right.
\end{equation*}
this system allows both $p$ and $q$ to be $\pi_1$ in a reduct over $\mathbb{Z}_5$;

So, nothing new can be obtained with the second identity being $p(x,y,x) \approx q(x,x,y)$.

\item if we have $p(x,y,x) \approx q(x,y,x)$ as the second identity, adding a third one from \eqref{system2} may give us the following (again, adding any two identities from  \eqref{system2} gives us the whole system \eqref{system2}):
\begin{equation*}
\left\{
\begin{array}{r}
p(x,x,y)\approx p(x,y,y)\\
p(x,y,x) \approx q(x,y,x)\\
q(x,x,y)\approx q(y,x,x)
\end{array}\right.
\end{equation*}
this system allows $p$ to be $\pi_1$ and $q$ to be $3x + 3z$ in a reduct over $\mathbb{Z}_5$;
\begin{equation*}
\left\{
\begin{array}{r}
p(x,x,y)\approx p(x,y,y)\\
p(x,y,x) \approx q(x,y,x)\approx q(x,x,y)
\end{array}\right.
\end{equation*}
this system allows both $p$ and $q$ to be $\pi_1$ in a reduct over $\mathbb{Z}_5$;
\begin{equation*}
\left\{
\begin{array}{r}
p(x,x,y)\approx p(x,y,y)\\
p(x,y,x) \approx q(x,y,x)\approx q(y,x,x)
\end{array}\right.
\end{equation*}
this system allows $p$ to be $\pi_1$ and $q$ to be $\pi_3$ in a reduct over $\mathbb{Z}_5$;

There are no more cases with $p(x,y,x) \approx q(x,y,x)$ as the second identity, except for the whole system \eqref{system2}.

\item if we have $p(x,y,x) \approx q(y,x,x)$ as the second identity, adding a third one from \eqref{system2} may give us the following (adding two identities gives the whole system, as before):
\begin{equation*}
\left\{
\begin{array}{r}
p(x,x,y)\approx p(x,y,y)\\
p(x,y,x) \approx q(y,x,x)\approx q(x,y,x)
\end{array}\right.
\end{equation*}
this system allows $p$ to be $\pi_1$ and $q$ to be $\pi_3$ in a reduct over $\mathbb{Z}_5$;
\begin{equation*}
\left\{
\begin{array}{r}
p(x,x,y)\approx p(x,y,y)\\
p(x,y,x) \approx q(y,x,x)\approx q(x,x,y)
\end{array}\right.
\end{equation*}
this system allows $p$ to be $\pi_1$ and $q$ to be $\pi_2$ in a reduct over $\mathbb{Z}_5$;
\begin{equation*}
\left\{
\begin{array}{r}
p(x,x,y)\approx p(x,y,y)\\
p(x,y,x) \approx q(y,x,x)\\
q(x,x,y)\approx q(x,y,x)
\end{array}\right.
\end{equation*}
this system allows $p$ to be $\pi_1$ and $q$ to be $3y + 3z$ in a reduct over $\mathbb{Z}_5$;

Once again we have obtained nothing new with $p(x,y,x) \approx q(y,x,x)$ as the second identity.
\end{itemize}
By this we have proved minimality of the system \eqref{system2}, with respect to implying omitting types 1 and 2.
\end{proof}

\noindent Further more, the following can be proved (simply by analyzing all possible cases): elimination of identities from the system \eqref{system1} gives us either a system equivalent to \eqref{system2}, up to the permutation of variables, or a system that holds in a reduct of a module over some finite ring, which means \eqref{system2} is the only system implying omitting types 1 and 2 obtainable from \eqref{system1} (and a proper subset of \eqref{system1}). We shall not provide the whole proof here, for it is too long, but only analyze two proper subsets of the system \eqref{system1} (the whole proof, however, can be found in \cite{nnmm}):
\begin{itemize}
\item[subset 1]
\begin{equation*}
\left\{
\begin{array}{r}
p(x,x,y)\approx p(x,y,y)\\
p(x,y,x)\approx q(y,x,x)\approx q(x,y,x) \approx q(x,x,y)
\end{array}\right.
\end{equation*}
This system implies nonexistence of $p$ (and $q$) in any reduct of a module over a finite ring, so it implies omitting types 1 and 2, but it is equivalent to system \eqref{system2}, and is obtainable from it by a permutation of variables of the term $q$ (we substitute $q(z,x,y)$ for $q(x,y,z)$).
\item[subset 2]
\begin{equation*}
\left\{
\begin{array}{r}
p(x,x,y)\approx p(x,y,x)\\
p(x,y,y)\approx q(y,x,x) \approx q(x,y,x)\approx q(x,x,y)
\end{array}\right.
\end{equation*}
This system allows $p$ and $q$ to be respectively $4x + y + z$, $2x + 2y + 2z$ in a reduct over $\mathbb{Z}_5$, so it does not describe omitting types 1 and 2.
\end{itemize}

 \vspace{0.3cm}

\noindent Up to this point we have started from \eqref{sys}, eliminated the first identity obtaining \eqref{system1}, and then eliminated yet another identity from \eqref{system1} obtaining \eqref{system2}, which is proven to imply omitting types 1 and 2, and to be the only (minimal) system with this property obtainable from \eqref{system1}, up to the permutation of variables. So, the next step would be analyzing what happens if we eliminate an identity other than the first one from \eqref{sys}.

\vspace{0.2 cm}

\noindent First we shall suppose the system obtained includes the identity $p(x,x,y)\approx x$ . It also has to include an identity on  both $p$ and $q$, as explained at the beginning of the current section. So, up to this point we have a subset of the system \eqref{sys} with the first identity being $p(x,x,y)\approx x$, and the second is an identity from \eqref{sys} on $p$ and $q$. There may be more identities from the system \eqref{sys} in this subset, which (the subset) we shall denote by  $\sigma$. We keep in mind $\sigma$ needs to hold in algebra $\mathbf{B}$ from example 1, so let us discuss on possible cases:
\begin{itemize}
\item in the identity on both $p$ and $q$ (the second identity in $\sigma$), left hand side must not be $p(x,x,y)$, for if this is the case, $q$ has to be at most a binary term in algebra $\mathbf{B}$, as well as $p$. This would allow us to define both terms in a reduct over $\mathbb{Z}_{5}$, as projections and/or binary terms (some of the terms $3x + 3y$, $3x + 3z$, $3y + 3z$). So, in the second identity of the system $\sigma$ left hand side is either $p(x,y,y)$ or $p(x,y,x)$.
\item if the terms $x$ and $p(x,x,y)$ do not occur at all in the rest of the system $\sigma$, then all the identities except for the first one include only some of the terms $p(x,y,x)$, $p(x,y,y)$, $q(x,x,y)$, $q(x,y,x)$, $q(y,x,x)$. Therefore we could define $p$ and $q$ to be $4x+2y$ and $2x + 2y + 2z$ respectively in a reduct over $\mathbb{Z}_5$. So, $p(x,x,y)$ or $x$ alone  must occur somewhere in the rest of the system $\sigma$.
\item if the term $p(x,x,y)$ (or $x$ alone) occurs in an identity on $p$ in the rest of the system, we shall obtain either $p(x,x,y)\approx x \approx p(x,y,y)$, or $p(x,x,y)\approx x \approx p(x,y,x)$. In both cases  $p$ has to be a projection map in $\mathbf{B}$, and because of the identity on both $p$ and $q$ (i.e. the second identity of the system $\sigma$), $q$ is at most a binary term in this algebra, but this allows $p$ and $q$ to be defined in a reduct over $\mathbb{Z}_{5}$, as explained earlier.
\item if the term $p(x,x,y)$ (or $x$ alone) occurs in an identity on $p$ and $q$ (i.e. on $x$ and $q$) we shall obtain one of the following identities: $p(x,x,y)\approx q(x,x,y) \approx x $, $p(x,x,y)\approx q(y,x,x) \approx x $, $p(x,x,y)\approx q(x,y,x) \approx x $. Again each of them   means $q$ is at most a binary term in $\mathbf{B}$, as well as $p$, so both are definable in a reduct over $\mathbb{Z}_{5}$.
\end{itemize}
By this it is proved that the system $\sigma$ can not describe omitting types 1 and 2.

\vspace{0.1 cm}

\noindent By elimination of identities from the system \eqref{sys}, we can also obtain a system that includes the identity $p(x,y,x) \approx x$ or $p(x,y,y) \approx x$ (of course, in each case it has to include an identity on both $p$ and $q$). Let us provide a brief overview on these systems:
\begin{itemize}
\item[case 1] Let $\sigma_1$ be a subset of the system \eqref{sys} with the first identity being $p(x,y,x) \approx x$ (the second identity in $\sigma_1$ is an identity on $p$ and $q$ from \eqref{sys}).  We can obtain an equivalent system  by a permutation of variables of the term $p$ (we substitute $p(x,y,z)$ for $p(x,z,y)$ in $\sigma_1$ ) but this is the system of the form $\sigma$, with the first identity being $p(x,x,y)\approx x$,  which we have already proved not to describe omitting types 1 and 2.
\item [case 2] Let $\sigma_2$ be a subset of the system \eqref{sys} with the first identity being $p(x,y,y) \approx x$ (the second identity in $\sigma_2$ is an identity on $p$ and $q$ from \eqref{sys}). Since the system $\sigma_2$ needs to hold in $\mathbf{B}$, $p$ is to be $\pi_1$ and $q$ at most a binary term in this algebra. This allows both of them to be defined in a reduct over $\mathbb{Z}_{5}$. Therefore $\sigma_2$ cannot describe omitting types 1 and 2.
\end{itemize}

\noindent We can conclude now that none of the systems obtained from \eqref{sys} including either of the identities $x \approx p(x,x,y)$, $x \approx p(x,y,x)$, $x \approx p(x,y,y)$ can describe omitting types 1 and 2.

\vspace{0.3 cm}

\noindent Let us now discuss the subsets of the system \eqref{sys} in which the first identity is one of the following three: $x \approx q(x,x,y)$, $x \approx q(x,y,x)$, $x \approx q(y,x,x)$. By examination of these systems (which is not presented here for it is too long and done in the same manner as already seen, but can be found in \cite{nnmm}), we come to the single system, up to the permutation of variables, that implies (and possibly describes) omitting types 1 and 2:
\begin{equation}
\left\{
\begin{array}{r}
x\approx q(x,y,x)\\
p(x,y,y)\approx p(x,y,x)\\
p(x,x,y)\approx q(x,x,y) \approx q(y,x,x)\label{newone}
\end{array}\right.
\end{equation}

\vspace{0.3 cm}

\subsection{both majority terms} \label{case3}
If both $p$ and $q$ have to be majority terms in $\mathbf{A}$ they satisfy the following in this algebra:
\begin{equation}
\begin{array}{r}
x\approx p(x,x,y)\approx p(x,y,x)\approx p(y,x,x)\approx q(y,x,x)\approx q(x,y,x)\approx q(x,x,y)
\label{system4}
\end{array}
\end{equation}
\noindent It is easily seen that $p$ and $q$ satisfying the system cannot exist in $\mathbf{B}$ (example 1), so some identities have to be eliminated. Notice that we have to keep at least one  identity having $x$ on the left, for if we do not, the remaining system holds in a reduct of a module over $\mathbb{Z}_5$ (we can define both $p$, $q$ to be $2x + 2y + 2z$).

Let us try eliminating the identity $p(x,x,y)\approx p(x,y,x)$. Now we have the following system that holds in $\mathbf{B}$:
\begin{equation}
\left\{
\begin{array}{r}
x\approx p(x,x,y)\\
p(x,y,x)\approx p(y,x,x)\approx q(y,x,x)\approx q(x,y,x)\approx q(x,x,y)
\label{system3}
\end{array}\right.
\end{equation}
\noindent We shall prove now that the system \eqref{system3} implies omitting types 1 and 2.
\begin{proof}
To prove that the system implies omitting types 1 and 2 it is sufficient to show that $p$, $q$ cannot exist in any reduct of a module over a finite ring. From the identities of the system it is easily seen that $q$ can only be a ternary term in any reduct (if it exists at all) i.e. $\alpha x + \beta y + \gamma z$ for $\alpha + \beta + \gamma =1$ and none of $\alpha, \beta, \gamma$ is zero. From the forth identity we obtain the following: $\alpha y + (\beta + \gamma)x = (\alpha + \gamma )x + \beta y$, and this implies $\alpha = \beta$. Then from the last identity we have: $(\alpha + \gamma)x + \alpha y = 2\alpha x + \gamma y$ , which gives $\alpha = \gamma$ also. Therefore, if $q$ exists in any reduct of a module over a finite ring, it has the form $\alpha x + \alpha y + \alpha z$, for $3\alpha = 1$. Then, from the third identity we have $p(x,y,z)\approx \alpha x + 2\alpha y$ (since $p$ can only be a binary term in any reduct), and therefore the second identity gives $\alpha = 2\alpha$ (this should hold in a finite ring mentioned), which implies $\alpha = 0$, and this is a contradiction. By this we have proved that $q$ (and consequently $p$) cannot exist in any reduct of a module over any finite ring, which means the system  \eqref{system3} implies omitting types 1 and 2.
\end{proof}

\vspace{0.3cm}

\noindent The system \eqref{system3} is also a minimal system implying omitting types 1 and 2, for any proper subset of \eqref{system3} holds in some reduct of a module over a finite ring (for some terms $p$ and $q$). This is proven the same way as minimality of the system \eqref{system2}, i.e. by analyzing all possible cases, and the proof is provided in \cite{nnmm}.

\vspace{0.3 cm}

\noindent If we return to the system \eqref{system4} and analyze other ways to eliminate identities in order to obtain a system that describes omitting types 1 and 2, we come to the following conclusion: if a system obtained (by elimination of identities from \eqref{system4}) holds in algebra $\mathbf{B}$, it is either equivalent to \eqref{system3} up to the permutation of variables, or it holds in a reduct of a module over some finite ring (once again, the examination of all the subsets of the system \eqref{system4} is provided in \cite{nnmm}).

\vspace{0.1cm}

\noindent We can conclude the following: from the system \eqref{system4} we can obtain a single system, up to the permutation of variables, which possibly describes omitting types 1 and 2 and that is the system \eqref{system3}.

\vspace{0.3 cm}

\noindent By this we have examined all possible forms of a system on $p$ and $q$ regarding the existence of terms in $\mathbf{A}$, and obtained three systems (that is \eqref{system2}, \eqref{newone}, \eqref{system3}), on variables $x$ and $y$, that imply and possibly describe omitting types 1 and 2. Before the conclusion there is yet another question left to examine -- can we obtain anything new from systems of identities on $p$ and $q$ with more than two variables?

\vspace{0.3 cm}
\subsection{systems of identities on $\mathbf{p}$ and $\mathbf{q}$ including more than two variables}
\noindent In this subsection we shall prove that nothing new can be obtained from systems of identities on $p$ and $q$ including more than two variables.

\vspace{0.1 cm}

\noindent Suppose we have a system on two ternary idempotent terms $p$ and $q$ that describes omitting types 1 and 2, and suppose there is an identity (or more of them) including more than two variables in this system. We shall denote this system by $\tau$ and discuss possible cases according to the number of variables:
\begin{itemize}
\item if there is an identity including six variables in $\tau$ , it can only be one of these two identities: $p(x,y,z) \approx p(u,v,w)$, $p(x,y,z) \approx q(u,v,w)$. In both cases we obtain the identity  $x \approx w$, which only holds in a trivial algebra (i.e. a trivial variety), so this case is impossible (assuming that $\tau$ describes omitting types 1 and 2).
\item if there is an identity including five variables in the system $\tau$, there are two possibilities:
\begin{itemize}
\item the identity mentioned can yield a trivial variety -- this happens there are $x,y,z$ on the left and $u,v$ on the right,  e.g. $p(x,y,z) \approx p(u,u,v)$, $p(x,y,z) \approx q(v,u,v)$, but also in these two cases: $p(x,y,z) \approx p(u,x,v)$, $p(x,y,z) \approx p(u,v,x)$ (if we substitute $x$ for $u$ and $v$ in the latter two identities, we shall obtain $p(x,y,z) \approx x$, which means $p$ has to be the first projection map, but that would give us  $x \approx u$, which yields a trivial variety).
\item the identity considered can imply that one of the terms $p$ and $q$, or both of them, must be a projection map (maps) -- this happens if the identity is one of the following: $p(x,y,z) \approx p(x,u,v)$, $p(x,y,z) \approx q(x,u,v)$, $p(x,y,z) \approx q(u,x,v)$, $p(x,y,z) \approx q(u,v,x)$. The identity $p(x,y,z) \approx p(x,u,v)$ implies that $p$ has to be a projection map, which exists in any algebra, so in this case we can substitute $x$ (or $y$ or $z$) for $p(x,y,z)$ in the system $\tau$, obtaining a system only on $q$, which cannot describe omitting types 1 and 2 (this has been proven in the subsection \ref{singleternary}). If both terms have to be projection maps then the system $\tau$ obviously can not describe omitting types 1 and 2.
 \end{itemize}
\item if there is an identity including four variables in the system $\tau$, there are three possibilities:
\begin{itemize}
\item the identity mentioned can yield a trivial variety, like in these cases: $p(x,y,z) \approx w$, $p(x,y,z) \approx p(w,w,x)$, $p(x,y,z) \approx p(w,x,w)$, $p(x,y,z) \approx p(w,x,x)$, $p(x,y,y) \approx p(w,w,z)$... Obviously none of these identities cannot occur in the system $\tau$.
\item the identity including four variables can imply that one of the terms has to be a projection map in any algebra, i.e.: $p(x,y,z) \approx p(x,w,w)$, $p(x,y,z) \approx q(w,w,x)$, $p(x,y,z) \approx p(x,w,x)$, $p(x,y,z) \approx q(w,x,x)$, $p(x,y,x) \approx p(z,w,x)$, $p(x,y,x) \approx q(z,w,x)$ ... If this were the case, we could substitute a single variable for one of the terms in $\tau$ (for example, if $p$ has to be the first projection map then we can substitute $x$ for $p(x,y,z)$ in the whole system), obtaining a system on a single ternary term which does not describe omitting types 1 and 2 (proved in the subsection \ref{singleternary}). Therefore none of the identities from above may occur in the system $\tau$.
\item the identity including four variables can imply that one of the terms has to be at most a binary term in any algebra, i.e.: $p(x,y,z) \approx p(w,x,y)$, $p(x,y,z) \approx p(x,w,y)$, $p(x,y,z) \approx q(y,x,w)$...(from the first identity we obtain $p(x,y,z) \approx p(x,x,y)\approx t(x,y)$, for some new binary term $t$). This means we can substitute a binary term $t(x,y)$ for $p(x,y,z)$ in the whole system $\tau$, obtaining a system on a binary and a ternary term ($t$ and $q$ respectively), and this is proven not to describe omitting types 1 and 2 (subsection \ref{bin_ter}). Therefore none of the identities from above may occur in the system $\tau$.
\end{itemize}
\item if there is an identity including three variables in the system $\tau$, these are the possible cases:
\begin{itemize}
\item there are three variables on one side of the identity , and only two of them on the other:$p(x,y,z) \approx p(x,x,y)$, $p(x,y,z) \approx q(x,y,y)$... In these cases we can substitute a new binary term $t(x,y)$ for the term $p$ in the whole system $\tau$, obtaining a system on a binary and a ternary term which cannot describe omitting types 1 and 2 (subsection \ref{bin_ter}). Therefore none of these identities may occur in the system $\tau$.
\item there are three variables on both sides of the identity: $p(x,y,z) \approx p(x,z,y)$, $p(x,y,z) \approx p(y,x,z)$, $p(x,y,z) \approx p(y,z,x)$... Each of these identities may occur in $\tau$. On the other side, $\tau$ can not include any of the identities $p(x,y,z) \approx q(x,z,y)$, $p(x,y,z) \approx q(y,z,x)$, $p(x,y,z) \approx q(y,x,z)$, etc, for if this were the case we could simply substitute $p$ for $q$ in the whole system obtaining a system only on $p$ which cannot describe omitting types 1 and 2 (subsection \ref{singleternary}). So, $\tau$ can include only some of the identities on $p$ with three variables on both sides.
\item there are two variables on each side of the identity: $p(x,y,y) \approx p(x,z,x)$, $p(x,y,x) \approx p(z,z,x)$... Each of these identities may occur in the system $\tau$. As for the identities on both $p$ and $q$ (such as $p(x,y,y) \approx q(x,z,x)$, $p(x,x,y) \approx q(x,z,z)$...), we can notice the following: in algebra $\mathbf{B}$ from example 1 both terms have to be at most binary, but if the system $\tau$ allows that, both terms can be defined in a reduct over $\mathbb{Z}_5$ (as at most binary also). This is impossible assuming that $\tau$ describes omitting types 1 and 2, so $\tau$ includes no identities on $p$ and $q$ with two variables on each side.
\end{itemize}

\vspace{0.1 cm}

\noindent So, if $\tau$ includes identities with more than two variables, they can only have three variables, and be of two kinds:
\begin{itemize}
\item identities on $p$ with $x$, $y$, $z$ on both sides
\item identities on $p$ with $x$, $y$ on one side and $x$, $z$ on another
\end{itemize}
\end{itemize}

\vspace{0.1 cm}

\noindent Now, the system $\tau$ has to hold in algebras $\mathbf{A}$ and $\mathbf{B}$, so if we substitute both $x$ and $y$ for $z$ in $\tau$ we shall obtain a system -- consequence (with more identities, but including only $x$ and $y$) that also holds in these two algebras. Let us denote this new system by $\tau_1$, and discuss what happens in a reduct of a module over a finite ring:
\begin{itemize}
\item identities on $p$ with $x$, $y$, $z$ on both sides
\begin{itemize}
\item if we substitute both $x$ and $y$ for $z$ in the identity $p(x,y,z) \approx p(x,z,y)$, we shall obtain the following two identities: $p(x,y,x) \approx p(x,x,y)$, $p(x,y,y) \approx p(x,y,y)$. The second identity is obviously a trivial one, and the first can hold in a reduct of a module over a finite ring if we define $p$ to be the first projection map, or a term $\alpha x+\beta y+ \beta z$, where $\alpha + 2\beta =1$. In both cases the identity $p(x,y,z) \approx p(x,z,y)$ holds in the same reduct.
\item if we substitute both $x$ and $y$ for $z$ in the identity $p(x,y,z) \approx p(z,x,y)$, we shall obtain the following two identities: $p(x,y,x) \approx p(x,x,y)$, $p(x,y,y) \approx p(y,x,y)$. These two hold in a reduct of a module over a finite ring if we define $p$ to be $\alpha x+\alpha y+ \alpha z$ (of course $3 \alpha = 1$), but this means that the identity $p(x,y,z) \approx p(z,x,y)$ also holds in this reduct.
\item the same holds for the identity $p(x,y,z) \approx p(z,y,x)$-- namely, by substituting $x$ and $y$ for $z$ we obtain two identities, $p(x,y,x) \approx p(x,y,x)$, $p(x,y,y) \approx p(y,y,x)$  . If we define $p$ (in any possible way) so that the two identities obtained hold in some reduct of a module over a finite ring, then the identity $p(x,y,z) \approx p(z,y,x)$ also holds in that reduct for the same term $p$.
\item it is easy to check that the same holds for the identities $p(x,y,z) \approx p(y,z,x)$, $p(x,y,z) \approx p(y,x,z)$.
\end{itemize}
\item identities on $p$ with $x$, $y$ on one side and $x$, $z$ on another
\begin{itemize}
\item it is easy to see that an identity like that does not hold in any reduct of a module over a finite ring if and only if the variable $z$ on the right side of the identity occurs in all the positions where $x$ is on the left (and perhaps in some more), e.g. $p(x,y,x) \approx p(z,x,z)$, $p(x,y,y) \approx p(z,z,x)$... For all of these identities (that do not hold in any reduct) holds the following: if we substitute both $x$ and $y$ for $z$, we shall obtain two identities such that both of them also cannot hold in any reduct (e.g. from the identity $p(x,y,x) \approx p(z,x,z)$ we would obtain the following two $p(x,y,x) \approx x$ , $p(x,y,x) \approx p(y,x,y)$. Now from the first one we have that $p$ would have to be one of the terms $\pi_1$, $\pi_3$, $\alpha x +\beta z$, but the second identity does not allow either of these).
\item according to the previous item, the following holds: if we have an identity on $p$ with $x$, $y$ on one side and $x$, $z$ on another, and by substituting both $x$ and $y$ for $z$ we obtain two identities that hold in some full idempotent reduct over a finite ring (for some term $p$), then the identity we started with (the one including $z$) also holds in that reduct for the same term $p$.
\end{itemize}
\end{itemize}

\vspace{0.1 cm}

\noindent We can now state the following: if the system $\tau_1$ (the system -- consequence,  obtained by substituting both $x$ and $y$ for $z$ in the system $\tau$) holds in some reduct of a module over a finite ring then the system $\tau$ also holds in that reduct. Since we want $\tau$ to describe omitting types 1 and 2, by substituting both $x$ and $y$ for $z$ in $\tau$ we need to  obtain a system with only $x$ and $y$ (i.e. $\tau_1$) that holds in $\mathbf{A}$ and $\mathbf{B}$ and does not hold in any reduct of a module. There are only three systems with only $x$ and $y$ that satisfy this, and these are the systems  \eqref{system2}, \eqref{newone}, \eqref{system3} obtained in the subsections \ref{cases12}, \ref{case3}  . Therefore $\tau_1$ has to contain one of these systems or actually be one of them. In either case, the system $\tau$ with three variables is a stronger condition compared to the obtained system $\tau_1$ with $x$ and $y$ only, so there is no need to consider it. In other words, we have just proved that nothing new can be obtained from systems of identities on $p$ and $q$ including more than two variables.

\vspace{0.3 cm}

\noindent Up to this point we can state the following:
\begin{fakt} If it is possible to describe omitting types 1 and 2 by two ternary terms $p$ and $q$, it can only be done by one or more of the systems \eqref{system2}, \eqref{newone}, \eqref{system3}.
\end{fakt}
\vspace{1cm}

\noindent We shall now provide an example of a finite algebra that generates a variety  omitting types 1 and 2 but does not realize either of the systems \eqref{newone}, \eqref{system3} (this example is due to Keith Kearnes).\\

\noindent $\mathbf{Example \  3}$\\
 \noindent Suppose the language consists only of a single ternary function symbol, $f(x,y,z)$.

  \vspace{0.1 cm}

  \noindent Let $\mathbf{C}=\langle \ \{\ \!0 \ ,\ 1 \} \ ,\ f^{\bf C}  \  \rangle$  be a 2-element algebra of this language where the basic operation $f^{\bf C}$ is a majority operation, $f^{\bf C}(x,x,y) \approx f^{\bf C}(x,y,x) \approx f^{\bf C}(y,x,x) \approx x $. Each ternary term operation in this algebra is either one of the projection maps or a majority term (this was already explained in the example 1).

   \vspace{0.1 cm}

   \noindent Furthermore, let $\mathbf{D}=\langle \ \{\ \!0 \ ,\ 1 \} \ ,\ f^{\bf D}  \  \rangle$  be a 2-element algebra of the same language where $f^{\bf D}$  is defined like this: $f^{\bf D} (x,y,z)\approx x \land y\land z$,  $\land$ being a semi--lattice meet operation.

    \vspace{0.1 cm}

    \noindent The product algebra $\mathbf{C} \times \mathbf{D}$ generates a variety that omits types 1 and 2 --  if we denote its basic operation by  $f^{{\bf C}\times{\bf D}}$, then the term operations $v(x,y,z) \approx  f^{{\bf C}\times{\bf
D}}(x,y,z)$ and $w(x,y,z,u) \approx f^{{\bf C}\times{\bf D}}(x,y,f^{{\bf C}\times{\bf D}}(x,z,u))$ satisfy the conditions of  Theorem 2.7. Let us now prove that the algebra  $\mathbf{C} \times \mathbf{D}$  realizes the system \eqref{system2}, but does not realize either of the systems \eqref{newone} or \eqref{system3}.
\begin{proof}
As mentioned before, let $f^{{\bf C}\times{\bf D}}$ be the basic operation in $\mathbf{C} \times \mathbf{D}$.\\

\vspace{0.3 cm}

 \noindent The terms $p(x,y,z) = f^{{\bf C}\times{\bf D}}(x,x,f^{{\bf C}\times{\bf D}}(x,y,z))$ and $q(x,y,z) = f^{{\bf C}\times{\bf D}}(x,y,z)$ satisfy the system \eqref{system2}.\\
 \vspace{0.3 cm}

 \noindent In the system \eqref{newone} there is a term $q(x,y,z)$ such that $x \approx q(x,y,x)$ and $q(y,x,x) \approx q(x,x,y)$. If $\mathbf{C} \times \mathbf{D}$ satisfies the identity $x \approx q(x,y,x)$ (i.e. there is such a term operation of this algebra), then algebra $\mathbf{D}$ satisfies the same (for the corresponding term operation $q$, of course). We can notice here that algebra $\mathbf{D}$ is in fact term--equivalent to a semi--lattice, which means  each of its term operations is just a meet of the variables involved. So, if $\mathbf{D}$ satisfies the identity $x \approx q(x,y,x)$, it means the term operation $q$ of $\mathbf{D}$ does not depend on its second variable. This also means that the corresponding term $q$ of the term algebra in question does not depend on its second variable (that is,  $y$ does not occur syntactically in the term $q(x,y,z)$ of the term algebra, for if it did then it would have to occur in the corresponding term operation $q$ of $\mathbf{D}$, since the only operation symbol $f$ of the language is interpreted as meet in $\mathbf{D}$ ). According to this, in algebra $\mathbf{C}$ the term operation  $q$ cannot depend on its second variable, so it has to be either the first or the third  projection map (this majority algebra does not have other binary term operations except for projections, example 1). But, $q$ should also satisfy $q(y,x,x) \approx q(x,x,y)$ in $\mathbf{C} \times \mathbf{D}$, therefore in $\mathbf{C}$, and this identity makes it impossible for $q$ to be either of these two projections. So, we can conclude that $\mathbf{C} \times \mathbf{D}$ does not realize the system \eqref{newone}.\\

 \vspace{0.1 cm}

 \noindent For the system \eqref{system3} holds the similar observation : in \eqref{system3} there is a term $p(x,y,z)$ such that $x \approx p(x,x,y)$ and $p(x,y,x) \approx p(y,x,x)$. If algebra $\mathbf{C} \times \mathbf{D}$ satisfies the identity $x \approx p(x,x,y)$ then  $\mathbf{D}$ satisfies the same, so $p$ cannot depend on its third variable in $\mathbf{D}$, therefore the third variable does not syntactically occur in the corresponding term $p(x,y,z)$ of the term algebra in question. Since the identity also has to hold in $\mathbf{C}$, the term operation $p$ has to be either the first or the second projection map in this algebra, but this is impossible because of the identity $p(x,y,x) \approx p(y,x,x)$. This means the algebra $\mathbf{C} \times \mathbf{D}$ does not realize the system \eqref{system3} either.
 \end{proof}
 \vspace{0.5 cm}
 \noindent We can conclude now that neither of the systems \eqref{newone}, \eqref{system3} characterizes omitting types 1 and 2, and also state a proposition:
 \begin{tvrdjenje}
 If it is possible to describe omitting types 1 and 2 by two ternary terms $p$ and $q$, it can only be done by the system \eqref{system2}.
 \end{tvrdjenje}

\noindent $\mathbf{Problem:}$ As we have mentioned in the abstract, it is not resolved whether the system \eqref{system2} actually describes omitting unary and affine types. Finding a counterexample for this system (i.e. a finite algebra that does not realize the system but generates a variety omitting types 1 and 2) would lead to conclusion that it is impossible to describe omitting unary and affine types by two ternary terms. On the other hand, perhaps it is possible to prove that any locally finite variety omitting types 1 and 2 realizes the system  \eqref{system2}, which would, of course, mean that it characterizes omitting types 1 and 2.

\vspace{3 cm}
\noindent $\mathbf{Acknowledgements:}$

\noindent I thank to Keith Kearnes for providing me with example 3.\\ I also thank to my Ph.D. adviser Petar Markovi\'c.

\end{document}